\documentclass[preprint,12pt]{elsarticle}
\usepackage{mathrsfs}
\usepackage{amsfonts}

\usepackage{amssymb}
\usepackage{amsmath}
\usepackage{tikz}
\usetikzlibrary{automata,positioning,arrows,shapes,chains}
\usepackage{ntheorem}
\usepackage{graphicx}
\usepackage{subfigure}
\usepackage{anyfontsize}
\usepackage{bookmark}
\usepackage{color}
\usepackage{algorithm}

\usepackage{hyperref} %interlinkage
\hypersetup{%hypertex=true,
            colorlinks=true,
            linkcolor=blue,
            anchorcolor=blue,
            citecolor=blue}

\usepackage[symbol]{footmisc}
\usepackage{stmaryrd}
\biboptions{numbers,sort&compress}

\newcommand{\normmm}[1]{{\left\vert\kern-0.25ex\left\vert\kern-0.25ex\left\vert #1 \right\vert\kern-0.25ex\right\vert\kern-0.25ex\right\vert}}

\newtheorem{theorem}{\textbf{Theorem}}

\newtheorem{definition}{\textbf{Definition}}
\newtheorem{lemma}{\textbf{Lemma}}

\newtheorem*{proof}{Proof.}

\newtheorem{remark}{\textbf{Remark}}

\numberwithin{equation}{section}

\journal{Journal of mathematical analysis
and applications}

\begin{document}

\begin{frontmatter}



\title{Improved Bohr radius for holomorphic mappings with values in higher dimensional complex space\tnoteref{t1}}
\tnotetext[t1]{This work is the results of the research
projects funded by National Natural Science
Foundation of China (12061035), Research Foundation of Jiangxi Science and Technology
Normal University of China (No.2021QNBJRC003).}

\author[1]{Luyi Zheng}
\ead{zengluyi2024@163.com}
\author[1]{Liangpeng Xiong\corref{cor1}}
\ead{lpxiong2016@whu.edu.cn}
\cortext[cor1]{Corresponding author}
\address[1]{School of Mathematics and Computer Science,
Jiangxi Science and Technology Normal University,
Nanchang 330038, People's Republic of China}

\begin{abstract}
In this article, we extend the generalized Bohr inequality and Bohr-Rogosinski type inequality of one-dimensional analytic functions to higher-dimensional holomorphic mappings. In section3, we wll extend the inequalities to the holomorphic mappings with values in the closure of the unit polydisc of the space $\mathbb{C} ^n, n\ge 2$. And we will extend the inequalities to the holomorphic mappings with value in the closure of the unit ball of a JB$^*$-triple in section 4, respectively. We all obtain the sharp radius.
\end{abstract}

\begin{keyword}
Bohr-Rogosinski radius \sep holomorphic mapping \sep homogeneous polynomial expansion \sep $\text{JB}^*$-triple \sep unit polydisc

\MSC32H02 \sep 32A30
\end{keyword}
\end{frontmatter}


\section{Introduction}\label{sec1}

Let $\mathcal{B} $ be the class of analytic functions in the unit disk $\mathbb{D}=\{z\in \mathbb{C} : |z|<1\}$  of the form  $f(\zeta)=\sum_{n=0}^{\infty}a_{n}\zeta^{n} $ such that $\left | f(\zeta) \right | < 1$ in $\mathbb{D}$. Famous Bohr's theorem states that if $f\in \mathcal{B}$, then the series $$\sum_{n=0}^{\infty}|a_{n}||\zeta|^{n}\leq 1$$ for $|\zeta|\leq\frac{1}{3}$.  The constant $\frac{1}{3}$ are sharp and called the Bohr radius for the class $\mathcal{B}$.

It is well-known that Bohr originally obtained the inequality for $\left| z \right| \le \frac{1}{6}$. Later, Wiener, Riesz and Schur improved the value $\frac{1}{6}$ to $\frac{1}{3}$ independently.
The research topic of Bohr radius has attracted the attention of numerous scholars, including multidimensional complex analysis, harmonic and quasi-conformal mappings, operator-valued functions,
and various subclasses of analytic functions. Numerous refinements and generalizations have
been established (see, e.g. \cite{ref1}, \cite{ref2}, \cite{ref5} \cite{ref7}, \cite{ref20}, \cite{ref27}).

Kayumov et al. [34] introduced the Bohr-Rogosinski radius which is as follows. If $f\in \mathcal{B}$, then for every $N\in \mathbb{N}$ , we have $R_{N}^{f}(z)\le 1$ when $|\zeta|=r<R_N$, where
	\begin{equation*}
		\mathit{R}_{N}^{f}(\zeta)=\left| f(\zeta) \right| + \sum_{n=N}^{\infty} \left| a_n \right| r^{n}
	\end{equation*}
	and the constant $R_N$ is sharp and it is the positive root of the equation
	\begin{equation*}
		2(1+r)r^N-(1-r)^2=0.
	\end{equation*}
Through the paper, we use the following notations:
\begin{equation*}
	\left \| f_0 \right \| ^ 2:=\sum_{n=1}^{\infty } \left | a_n \right | ^2r^{2n}, |\zeta|=r,
\end{equation*}
where $\mathit{f}(\zeta)=\sum_{n=0}^{\infty } a_n\zeta^n$ and $\mathit{f}_0(\zeta)=\mathit f(\zeta)-\mathit f(0)$.

In many studies, they have improved Bohr and Bohr-Rogosinski inequality of the family $\mathcal{B}$ by using $\left \| f_0 \right \| ^ 2$. The improved Bohr-Rogosinski theorem (Theorem D) and improved Bohr theorem (Theorem C) for the family $\mathcal{B}$ were proved by Liu-Liu-Ponnusamy \cite{ref32} and Ponnusamy-Vijayakumar-Wirths \cite{ref35}, respectively.

\begin{theorem}\cite{ref35}
	Suppose that $\mathit{f}(z)=\sum_{n=0}^{\infty} \left| a_{n} \right| r^{n} \in \mathcal{B}$. Then
	\begin{align*}
		\sum_{n=0}^{\infty}|a_n| r^n+(\frac{1}{1+|a_0|}+\frac{r}{1-r})\sum_{n=1}^{\infty}|a_n|^2r^{2n} \le 1, \text{for} \left|z\right |=r \le r_0=\frac{1}{2+\left|a_0 \right|}.
	\end{align*}
	and the numbers $\frac{1}{2+\left | a_0 \right |}$ and $\frac{1}{1+\left | a_0 \right |}$ cannot be improved.
\end{theorem}

\begin{theorem}\cite{ref32}
	Suppose that $\mathit{f}(z)=\sum_{n=0}^{\infty} \left| a_{n} \right| r^{n} \in \mathcal{B}$. Then
	\begin{align*}
		\left| \mathit{f}(z) \right|+\sum_{n=N}^{\infty} \left| a_n \right| r^{n} +(\frac{1}{1+\left | a_0 \right | } +\frac{r}{1-r} )\left \| f_0 \right \| ^ 2 \le 1,\\
		\text{for} \left| z \right | =r \le r_0=\frac{2}{3+\left | a_0 \right |+\sqrt{5}(1+a_0)}.
	\end{align*}
	The radius $r_0$ is the best possible and $r_0>\sqrt{5}-2$.
\end{theorem}
Subsequent researchers extended Bohr's inequality from the unit disk to polydisks and more general bounded domains in complex Euclidean spaces \cite{ref3,ref4}, as well as to vector-valued settings \cite{ref34}, which laid the foundation for the systematic study of Bohr-type inequalities.

As commutative bounded symmetric domains, polydisks serve as the most fundamental setting for the multidimensional generalization of the Bohr radius. Dineen and Timoney \cite{ref17} obtained the Bohr radius of the polydisc tends to zero through qualitative analysis. On this basis, Boas and Khavinson \cite{ref8} through systematic research on the Bohr radius $K_n$ on the n-dimensional polydisk and established the two-sided estimate
    \begin{equation*}
       \frac1{3^n}\le K_n\le 2\frac{\log n}{n}.
    \end{equation*}
Obviously, as the dimension $n \to \infty$, the Bohr radius $K_n$ tends to zero. This reveals the fundamental difference between one-dimensional and high-dimensional Bohr phenomena. In subsequent studies, this estimate was continuously improved, see e.g.\cite{ref14,ref16}. And ultimately by Bayart, Pellegrino, and Seoane-Sepúlveda \cite{ref6} determined the exact asymptotics $K_n\sim \frac{\log n}{n}$. Nevertheless, polydisks represent only the reducible, commutative class of bounded symmetric domains, while the four classical families of irreducible bounded symmetric domains in the Cartan classification, together with two exceptional domains, cannot be fully described by the Euclidean dimension alone. A unified algebraic framework is therefore required to treat all symmetric domains simultaneously.

On account of the open unit ball of every $\text{JB}^*$-triple is a bounded symmetric domain; conversely, by Kaup's Riemann mapping theorem, every bounded symmetric domain in a complex Banach space is biholomorphically equivalent to the open unit ball of a $\text{JB}^*$-triple, which is unique up to isometric isomorphism \cite{ref30}. Hence the theory of $\text{JB}^*$-triple systems can provide a standard axiomatic framework that can handle all symmetric domains simultaneously. For the definition and the properties of a $\text{JB}^*$-triple, see e.g.\cite{ref12,ref13}.

Based on this framework, Bohr-type problems have been generalized to the unit balls of $\text{JB}^*$-triples(see e.g.\cite{ref15}, \cite{ref21}, \cite{ref24}, \cite{ref25} for the finite-dimensional case and \cite{ref15}, \cite{ref22}, \cite{ref23}, \cite{ref25}, \cite{ref26} for infinite-dimensional case). Especially in Hamada et al.'s latest research\cite{ref25}, they directly obtained the Bohr–Rogosinski radius for holomorphic mappings on the unit ball of a complex Banach space with values in a higher dimensional complex Banach space and the Bohr–Rogosinski radius for holomorphic mappings with values in the closure of the unit ball of a $\text{JB}^*$-triple. And they also obtained the Bohr-Rogosinski radius for hlomorphic mappings with values in the closure of the unit polydisc of ${\mathbb{C}}^n, n \ge 2$ and in the closure of the unit ball of a $\text{JB}^*$-triples in following theorems.
\begin{theorem}\cite{ref25}
	Let $B_X$ be the unit ball of a complex Banach space $X$. Let $a=(a_1,\dots,a_n)\in\overline{\mathbb{U}}^n$ and let $F$ be a holomorphic mapping from $B_X$ to $\overline{\mathbb{U}}^n$ with
	\begin{equation*}
		F(z)=a+\sum_{s=1}^{\infty}P_s(z),\quad z\in B_X,
	\end{equation*}
	where $P_s(z)=\frac{1}{s!}D^sF(0)(z^s)$. Assume that $|a_j|=\|a\|_{\infty}$ for all $j\in\{1,2,\dots,n\}$. Then for $m,N\in\mathbb{N}$ and $p\in(0,2]$, we have
	\begin{equation*}
		\|F(v(z))\|_{\infty}^p+\sum_{s=N}^{\infty}\|P_s(z)\|_{\infty}\leq 1.
	\end{equation*}
	for $\|z\|=r\leq R_{m,N}^p$, where $v:B_X\to B_X$ is a Schwarz mapping having $z=0$ as a zero of order $m$, and $R_{m,N}^p$ is the unique root in $(0,1)$ of the equation
	\begin{equation*}
		p\left(\frac{1-r^m}{1+r^m}\right)-\frac{2r^N}{1-r}=0.
	\end{equation*}
	The constant $R_{m,N}^p$ cannot be improved.
\end{theorem}
\begin{theorem}\cite{ref25}
	Let $B_X$ be the unit ball of a complex Banach space $X$ and let $\mathbb{B}_Y$ be the unit ball of a $\text{JB}^*$-triple $Y$. Let $F$ be a holomorphic mapping from $B_X$ to $\mathbb{B}_Y$ with
	\[
	F(z) = a + \sum_{s=1}^{\infty} P_s(z),\quad z\in B_X,
	\]
	where $P_s(z) = \frac{1}{s!} D^s F(0)(z^s)$. Assume that $a\in \mathbb{B}_Y$. Then for $m, N\in \mathbb{N}$ and $p\in(0,2]$, we have
	\[
	\|F(v(z))\|^p + \sum_{s=N}^{\infty}\frac{\|D\varphi_a(a)P_s(z)\|}{\|D\varphi_a(a)\|} \le 1
	\]
	for $\|z\| = r \le R_{m,N}^p$, where $v\colon B_X\to B_X$ is a Schwarz mapping having $z=0$ as a zero of order $m$, $\varphi_a\in \operatorname{Aut}(\mathbb{B}_Y)$ such that $\varphi_a(a)=0$ and $R_{m,N}^p$ is the unique root in $(0,1)$ of the equation (2.2). The number $R_{m,N}^p$ cannot be improved.
\end{theorem}
Therefore, a natural question arises: Can the high-dimensional Bohr and Bohr-Rogosinski inequalities given by Theorem E and Theorem F be generalized by using $\left \| f_0 \right \| ^ 2$?
\par
In this paper, we will extend the Bohr and Bohr-Rogosinski type inequalities given by Theorem E and Theorem F to holomorphic mappings with values in the closure of the unit polydisc of ${\mathbb{C}}^n, n \ge 2$ in section 3. And in section 4, we will extend them to holomorphic mappings with value in the closure of the unit ball of a \(JB^*\)-triples. We all obtained the sharp radius.

\section{Preliminaries}

If $X$ and $Y$ be complex Banach spaces with norm $\left\|\bullet\right\|$.
\begin{definition}
	Let $k \in \mathbb{N}$. If there exsits a $k$-linear mapping $u$ from $X^k$ to $Y$ which for any $x \in X$ has
	\begin{equation*}
		P(x)=u(x,\cdots, x).
	\end{equation*}
	then we call $P:X \to Y$ is a homogeneous polynomial of degree $k$.	
\end{definition}

Note that if $P$ is a homogeneous polynomial of degree $k$, then $P(\lambda X)=\lambda ^kP(x)$ for $x \in X $ and $\lambda  \in \mathbb{C}$.

In this paper, the degree of a homogeneous polynomial will be indicated by a subscript. Namely, if $P_k$ is a homogeneous polynomial, then $P_k$ has the degree $k$. It's worth noting that if $P_k$ is a $k$-homogeneous polynomial from $X$ to $Y$, then exsits a unique symmetric $k-$linear mapping $u$ with
\begin{equation*}
	P(x)=u(x,\cdots, x).
\end{equation*}

Let $\mathcal{D} \in X$, $F$ is a homogeneous polynomial from $\mathcal{D}$ to $Y$. Then for any $z \in \mathcal{D} $, we denote $D^kF(z)$ is the $k$-Fr$\acute{e}$chet derivative of $F$ at $z$. If $\mathcal{D}$ contains the origin, then any holomorphic mapping $F:\mathcal{D} \to Y$ can be expanded into the series:
\begin{equation}
	F(z)=\sum_{n=0}^{\infty} \frac{1}{k!} D^kF(0)(z^k)
\end{equation}

Moreover, if $ \frac{1}{k!} D^kF(0)(z^k)$ is a$k$-homogeneous polynomial, then we use the notation $P_k(z)= \frac{1}{k!} D^kF(0)(z^k)$ alone this paper. Note that if $\mathcal{D}$ is a bounded balance domain in a complex Banach space X, and $F(\mathcal{D})$ is bounded, then for any $r \in \left ( 0,1 \right ) $ has(1.2) be converged in $r\mathcal{D}$.

\begin{definition}
	Let $F$ is a holomorphic mapping from $\mathcal{D}$ to $Y$. For any $k \in \mathbb{N} $, if there has
	\begin{equation*}
		F(0)=0,DF(0)=0,\dots ,D^{k-1}F(0)=0 \quad but \quad D^kF(0)\ne0,
	\end{equation*}
	then we called $z=0$ is a zero of order $k$ of $F$.
\end{definition}

\begin{lemma}\cite{ref25}
	Let $B_X$ and $B_Y$ are the unit balls of the Banach spaces $X$ and $Y$,respectively. We called the holomorphic mapping $v:B_X \to B_Y$ with $v(0)=0$ as Schwarz mapping.
	If $z=0$ is the zero of order k of $v$, then there has estimate holds
	\begin{equation}
		\left \|v(z) \right \|_Y \le \left \| z \right \| ^k_X,\quad z\in B_X.
	\end{equation}
\end{lemma}

\begin{definition}
	Let $L\left ( X,\mathbb{C}  \right ) $ is a set of continuous linear operators from $X$ to $\mathbb{C}$. For any $x\in X\setminus \left \{ 0 \right \} $, let
	\begin{equation*}
		T(x)={l_x \in L\left ( X,\mathbb{C}  \right ):l_x(x)=\left \| x \right \|,\left \| l_x \right \|=1}.
	\end{equation*}
We can prove that this set is non-empty by using the Hahn-Banach theorem.
\end{definition}

The classical Schwarz lemma sates that for any holomorphis mappings $f$ for $\mathbb{D}$ to $\mathbb{D}$ with $f(0)=0$, then for all $z \in \mathbb{D}$ has $\left | f(z) \right | \le \left | z \right | $\cite{ref19}. Lindel$\ddot{o}$f removing the assumption that "the origin is a fixed point",then the classical Schwarz lemma (\cite{ref31}, Proposition 2.2.2) was improved. And based on Lindel$\ddot{o}$f's conclusion, Chen et al.\cite{ref11} obtained the Schwarz lemma between the unit spheres of complex Banach spaces as follows:

\begin{lemma}\cite{ref11}
	Suppose that $B_X$ and $B_Y$ are the unit balls of complex Banach spaces $X$ and $Y$,respectively. Let $f:B_X \to B_Y$ be a holomorphic mapping. Then
	\begin{equation}
		\left \| f(z) \right \|_Y \le \frac{\left \| f(0) \right \|_Y+\left \| z \right \|_X }{1+\left \| f(0) \right \|_Y\left \| z \right\|_X }, for z \in B_X.
	\end{equation}
	This estimate is sharp with equality for each value of $\left \| f(0) \right \|_Y$ and for each $z \in B_X$.
\end{lemma}

\begin{lemma}\cite{ref35}
Suppose that $f \in \mathcal{B}$ and $f(z) = \sum_{n=0}^{\infty} a_n z^n$. Then the following inequalities hold.
\begin{enumerate}
	\item[(a)] $|a_{2n+1}| \leq 1 - |a_0|^2 - \cdots - |a_n|^2,\ n = 0,1,\dots$,
	\item[(b)] $|a_{2n}| \leq 1 - |a_0|^2 - \cdots - |a_{n-1}|^2 - \frac{|a_n|^2}{1+|a_0|},\ n = 1,2,\dots$.
\end{enumerate}
Further, to have equality in (a) it is necessary that $f$ is a rational function of the form
	\begin{equation}
		f(z)=\frac{a_0+a_1z+\cdots +a_nz^n+\epsilon z^{2n}}{1+(\overline{a_n}z^n)+\cdots +\overline{a_0}z^{zn+1})\epsilon } ,\quad  \left| \epsilon \right|=1,
	\end{equation}
	and to have equality in (b) it is necessary that $f$ is a rational function of the form
	\begin{equation}
		f(z)=\frac{a_0+a_1z+\cdots +\frac{a_n}{1+\left | a_0 \right | } z^n+\epsilon z^{2n}}{1+(\frac{\overline{a_n}}{1+\left | a_0 \right | } z^n)+\cdots +\overline{a_0}z^{2n})\epsilon },\quad  \left| \epsilon \right|=1,
	\end{equation}
where the condition $a_0\overline{a_n}^2 \epsilon$ is non-positive real..
\end{lemma}

By performing a simple calculation using Lemma 3, we can obtain

\begin{lemma}\cite{ref35}
	Suppose that $f \in \mathcal{B}$ and $f(z) = \sum_{n=0}^{\infty} a_n z^n$. Then the inequality
	\begin{equation}
		M_f(r) \le \left | a_0 \right |+\frac{r}{1-r}(1- \left | a_0 \right |^2)-(\frac{1}{1+\left | a_0\right |}+\frac{r}{1-r})\left \| f_0 \right \|^2
	\end{equation}
	is valid for $r \in [0,1)$. Equality is attained for $f(z)=1,\ z\in\mathbb{D}$.
\end{lemma}

\begin{remark}
	For (6) we consider that $a_0=0$ and obtain that
	\begin{equation*}
		\sum_{n=1}^{\infty}\left | a_n\right |r^n \le \frac{1}{1-r}(r-\|f_0\|^2).
	\end{equation*}
\end{remark}
\begin{lemma}\cite{ref32}
	For any $f\in \mathcal{B} , r\in [0,1), t=\frac{N-1}{2}$, then there has the inequality
	\begin{align}
		S(r):=\sum_{n=N}^{\infty} \left | a_n \right | r^n
		&+\operatorname{sgn}(t)\sum_{n=1}^{t} \left | a_n \right |^2\frac{r^N}{1-r} +\left(\frac{1}{1+\left | a_0 \right | }+\frac{1}{1-r}\right)\sum_{n=t+1}^{\infty}\left |a_n  \right |^2r^{2n} \notag\\
		&\le (1-\left |a_0  \right |^2 )\frac{r^N}{1-r}
	\end{align}
holds.
\end{lemma}

\begin{lemma}\cite{ref10}
	Let $m\in \mathbb{N}$ and $p\in(0,2]$. For $a\in[0,1]$, consider
	\begin{equation*}
	a\mapsto D_{p,m}(a)=\left[\left(\frac{a+r^{m}}{1+ar^{m}}\right)^{p}-1\right]\varphi_{0}(r)+(1-a^{2})N(r),
\end{equation*}
	where $\varphi_{0}(r)$ and $N(r)$ are some nonnegative continuous functions defined on $[0,1)$. Also, suppose that
	\begin{equation*}
	\Psi_{p,m}(r)=p\left(\frac{1-r^{m}}{1+r^{m}}\right)\varphi_{0}(r)-2N(r)
\end{equation*}
	and $R:=R(m,p)$ is the minimal positive root in $(0,1)$ of the equation $\Psi_{p,m}(r)=0$. If $\Psi_{p,m}(r)\ge 0$ for $0\le r\le R$, then $D_{p,m}(a)\le 0$ for $0\le r\le R$.
\end{lemma}

Let $\mathbb{B}_Y$ is the unit ball of a $\mathrm{JB}^*$-triple $Y$. Let $\operatorname{Aut}(\mathbb{B}_Y)$ denote the set of biholomorphic mappings of $\mathbb{B}_Y$ onto itself. Then $\operatorname{Aut}(\mathbb{B}_Y)$ is a transitive group.

\begin{lemma}\cite{ref22}
	Let $F\colon \mathbb{U}\to \mathbb{B}_Y$ be a holomorphic mapping with
	\begin{equation*}
	F(\zeta) = a + \sum_{s=1}^{\infty} P_s(\zeta),\quad \zeta\in \mathbb{U},
    \end{equation*}
	where $P_s(\zeta) = \frac{1}{s!} D^s F(0)(\zeta^s)$. Assume that $a\in \mathbb{B}_Y$. Then, we have
	\begin{equation*}
	\frac{\|D\varphi_a(a)P_s(\zeta)\|}{\|D\varphi_a(a)\|} \le 1-\|a\|^2,\quad s\ge 1,\ \zeta\in \mathbb{U},
    \end{equation*}
	where $\varphi_a\in \operatorname{Aut}(\mathbb{B}_Y)$ such that $\varphi_a(a)=0$.
\end{lemma}

\section{Improved Bohr-Rogosinski inequalities for holomorphic mappings $F$: $B_X\rightarrow \mathbb{U}^n$.}\label{sec3}

In this section we have improved the inequalities established in Theorem E and Theorem F by using the $\left \| f_0 \right \| ^ 2$. First, we improve Theorem E and obtain the following theorem:
	\begin{theorem}
	Let $B_X$ is a unit ball of the complex banach space $X$.  Let $a=(a_1,a_2,\dots,a_n)\in\mathbb{U}^n$ and let $F$ be a holomorphic mapping from $B_X $ to $\overline {\mathbb{U}}^n$  with
	\begin{align*}
		F(z)=a+\sum_{s=1}^{\infty} P_s(z),\quad z \in B_X,
	\end{align*}
	where $P_s(z)=\frac{1}{s!} D^sF(0)(z^s)$.
	Assume that $\left |a \right |=\| a \|_\infty$ for all $j=1,\dots ,n$. Then for any $\left \| z \right \|  =r \le R_1$, we have
	\begin{align}
		\left \| a \right \| _\infty +\sum_{s=1}^{\infty} \left \| P_s(z) \right \| _\infty+\bigg(\frac{1}{1+\left \| a \right \| _\infty} +\frac{r}{1-r}\bigg)\sum_{s=1}^{\infty} \left \| P_s(z) \right \| _\infty^2 \le 1.
	\end{align}
	The constant $R_1=\frac{2}{1+\left \| a \right \| _\infty}$ cannot be improved.
    \end{theorem}
    \begin{proof}
	Let $\zeta_0=\frac{z}{\left \| z \right \| } \in \partial B_X$ be fixed and let $f_j(\zeta)=F_j(\zeta z_0),\,\,\zeta\in \mathbb{U}$, where $F=(F_1, F_2, \dots, F_n)$.
	Then $f_j\in H(\mathbb{U}, \overline{\mathbb{U}})$ and
	\begin{equation*}
		f_j(\zeta)=a_j+\sum_{s=1}^{\infty} (P_s)_j(\zeta_0)\zeta^s, \quad\zeta\in \mathbb{U}.
	\end{equation*}
For each fixed $j$ and $s$, there exists $w\in \partial_0 B_X$ such that $|(P_s)_j(\zeta_0)|\leq |(P_s)_j(w_0)|$.
Therefore, for each $j$, we have
	\begin{equation*}
	|(P_s)_j(\zeta_0)|\leq |(P_s)_j(w_0)|\le 1-|(P_0)_j(w_0)|^2=1-\left | a_j \right | ^2=1-\left \| a \right \| ^2_{\infty}.
	\end{equation*}
	This implies that
	\begin{equation}
		\left \| P_s(\zeta_0) \right \| \le 1-\left \| a \right \| ^2_{\infty}.
	\end{equation}
	By the Lemma 4, for each $j$, we have
	\begin{align}
		\left | a_j \right | &+\sum_{s=1}^{\infty} \left | (P_s)_j(\zeta_0)) \right | \left |\zeta\right | ^s \notag \\
		&\le \left | a_j \right | +(1-\left | a_j \right | ^2)\frac{\left |\zeta \right | }{1-\left |\zeta\right | } -\Bigg(\frac{1}{1+\left | a_j \right | }
+\frac{\left |\zeta \right | }{1-\left |\zeta\right |}\Bigg)\sum_{s=1}^{\infty} \left | (P_s)_j(z_0)) \right |^2 \left |\zeta\right | ^{2s} \notag \\
		&\le \left \| a \right \|_{\infty}  +(1-\left \| a \right \|_{\infty}  ^2)\frac{\left |\zeta\right | }{1-\left |\zeta\right | } -\Bigg(\frac{1}{1+\left \| a \right \|_{\infty}  } +\frac{\left |\zeta\right | }{1-\left |\zeta\right |}\Bigg)\sum_{s=1}^{\infty}\left |(P_s)_j(z_0) \right|^2\left |\zeta\right |^{2s}.
	\end{align}
() implies that
\begin{align}
\|a\|_\infty &+\sum_{s=1}^{\infty} \| (P_s)(\zeta_0))\|_\infty \left |\zeta\right | ^s \notag \\
&\le \left \| a \right \|_{\infty}  +(1-\left \| a \right \|_{\infty}  ^2)\frac{\left |\zeta\right | }{1-\left |\zeta\right | } -\Bigg(\frac{1}{1+\left \| a \right \|_{\infty}  } +\frac{\left |\zeta\right | }{1-\left |\zeta\right |}\Bigg)\sum_{s=1}^{\infty}\|(P_s)(\zeta_0)\|_\infty^2\left |\zeta\right |^{2s}.
	\end{align}
Setting $\zeta=\|z\|=r<1$ and let $x=\left \| a \right \|_{\infty}\in [0,1)$, then by using estimates (9), (10), we have
	\begin{align*}
		\left \| a \right \|_{\infty}&+\sum_{s=1}^{\infty}\left \| P_s(r\zeta_0) \right \|_{\infty}+\Bigg(\frac{1}{1+\left \| a \right \|_{\infty}} +\frac{r}{1-r}\Bigg) \sum_{s=1}^{\infty}\left \| P_s(r\zeta_0) \right \|_{\infty}^2\\
&=\|a\|_{\infty}+\sum_{s=1}^{\infty}\left \| P_s(\zeta_0) \right \|_{\infty}r^s+\Bigg(\frac{1}{1+\left \| a \right \|_{\infty}} +\frac{r}{1-r}\Bigg) \sum_{s=1}^{\infty}\| P_s(\zeta_0)\|_{\infty}^2r^{2s}\\
		&\le \left \| a \right \|_{\infty}+(1-\left \| a \right \|_{\infty}^2)\frac{r}{1-r}\\
		&=1+\frac{(-x^2-x+2)r+(x-1)}{1-r}.
	\end{align*}
	Let $\Psi_x (r)=(-x^2-x+2)r+(x-1)$. The inequality holds $\le 1$ if and only if $\Psi_x (r) \le 0$. Because of
	\begin{equation*}
		\frac{\partial \Psi_x (r)}{\partial r} = -x^2-x+2 \ge 0, \quad x\in [0,1].
	\end{equation*}
	It means that $\Psi_x (r)$ is monotonically increasing on $r \in [0,1)$, so we have
	\begin{equation*}
		\Psi(r) \le \Psi(\frac{1}{2+\left \| a \right \|_{\infty} } ) =1.
	\end{equation*}
	Hence,we obtain $R_1=\frac{1}{2+\left \| a \right \|_{\infty} }$.
\par
	Next, we will show that the number $R_1=\frac{1}{2+\left \| a \right \|_{\infty} }$ is optimal.
	Let $z_0\in \partial B_X$ and $\mathbf{u}=(u_1,\dots,u_n)\in \partial \mathbb{U}^n$ with $|u_1|=\cdots=|u_n|=1$ be fixed. For $\lambda\in(0,1)$, let
	\begin{equation*}
	F(z)=f(l_{z_0}(z))\mathbf{u},\quad z\in B_X,
    \end{equation*}
	where
	\begin{equation*}
	f(\zeta)=\frac{\lambda-\zeta}{1-\lambda\zeta},\quad \zeta\in \mathbb{U}
    \end{equation*}
	and $l_{z_0}\in T(z_0)$. Then we have
	\begin{align*}
		\left \| a \right \| _\infty &+\sum_{n=1}^{\infty} \left \| P_s(z) \right \| _\infty+(\frac{1}{1+\left \| a \right \| _\infty} +\frac{r}{1-r} )\sum_{n=1}^{\infty} \left \| P_s(z) \right \| _\infty^2\\
		&=\lambda+(1-\lambda^2)\frac{r}{1-\lambda r}+(\frac{1}{1+\lambda} +\frac{r}{1-r} )(1-\lambda^2)^2\frac{r^2}{1-\lambda^2r^2} \\
		&=1+\frac{(1-\lambda)[(2+\lambda)r-1]}{1-r}
	\end{align*}
	Obviously, the inequality will $>$ 1 when $r > \frac{1}{2+\lambda} $.Thus, $R_1$ is optimal. This completes the proof.
\end{proof}

We now already know that through replacing $\mathit{f}(z)$ with $\mathit{f}(0)$, the Bohr inequality can be optimized. We replaced $\left \| a \right \| _\infty $ with $\left \| F(v(z)) \right \| _\infty^p$ for the (8). And make the same method improvement to the Bohr-Rogosinski type inequality, we can obtain the following theorem:
\begin{theorem}
	Let $B_X$ is a unit ball of the complex banach space $X$.  Let $a=(a_1,a_2,\dots,a_n)\in\mathbb{U}^n$ and let $F$ be a holomorphic mapping from $B_X $ to $\overline {\mathbb{U}}^n$  with
	\begin{align*}
		F(z)=a+\sum_{s=1}^{\infty} P_s(z),\quad z \in B_X,
	\end{align*}
	where $P_s(z)=\frac{1}{s!} D^sF(0)(z^s)$.
	Assume that $\left |a_j \right |=\|a\|_\infty$ for all $j=1,\dots ,n$. Then for any $\left \| z \right \|  =r \le R_2$, we have
	\begin{align}
  \|F(v(z))\|_\infty^p+\sum_{n=1}^{\infty} \|P_s(z)\|_\infty+\Bigg(\frac{1}{1+\|a\|_\infty}+\frac{r}{1-r}\Bigg)\sum_{n=1}^{\infty}\|P_s(z)\|_\infty^2\le 1.
	\end{align}
	Where $v(z): B_X \to B_X$ is a Schwarz mapping having $z=0$ as a zero of order m, and constant $R_2$ is the unique root in $(0,1)$ of the equation
	\begin{equation}
	p\Bigg(\frac{1-r^m}{1+r^m}\Bigg)-2\frac{r}{1-r}=0.
	\end{equation}
	The constant $R_2$ cannot be improved.
\end{theorem}
\begin{proof}
	Let $z_0=\frac{z}{\left \| z \right \| } \in \partial B_X$ be fixed.
	From (), we have
	\begin{equation}
		\left \| P_s(z_0) \right \| \le 1-\left \| a \right \| ^2_{\infty}.
	\end{equation}
	Let $x=\left \| a \right \|_{\infty}\in [0,1)$ and $\left \| z \right \|=r<1$. By using Lemma 2, we have
\begin{align}
\|F(v(z))\|_{\infty}^p
&\le \Bigg(\frac{\left \| F(0) \right \|_\infty+\|v(rz_0)\|}{1+\|F(0)\|_\infty\|v(rz_0)\|}\Bigg)^p\notag\\&\leq\Bigg(\frac{\left \|a\right \|_\infty+\|rz_0)\|^m}{1+\|a\|_\infty \|rz_0\|^m}\Bigg)^p\notag\\&=\Bigg(\frac{\left \|a\right \|_\infty+r^m}{1+\|a\|_\infty r^m}\Bigg)^p\notag\\&=\Bigg(\frac{x+r^m}{1+x r^m}\Bigg)^p.
\end{align}
By using Lemma, the estimates () and (), we have
\begin{align*}
		\left \| F(v(rz_0)) \right \|_{\infty}^p &+\sum_{s=1}^{\infty}\left \| P_s(rz_0) \right \|_{\infty}+\Bigg(\frac{1}{1+\left \| a \right \|_{\infty}} +\frac{r}{1-r}\Bigg)\sum_{s=1}^{\infty}\left \| P_s(rz_0) \right \|_{\infty}^2 \\
		&\leq \Bigg(\frac{x+r^m }{1+xr^m}\Bigg)^p+(1-x^2)\frac{r}{1-r}\\
		&=1+\Phi_{m}^p(r),
	\end{align*}
	which is less than or equal 1 provided $\Phi _{m}^p(r)$ is less than or equal 0, where
	\begin{equation}
	\Phi _{m}^p(r)=\Bigg(\frac{x+r^m }{1+xr^m}\Bigg)^p+(1-x^2)\frac{r}{1-r}-1.
	\end{equation}
	Hence, by Lemma 6, let $x=a, \varphi _0(r)\equiv 1, N(r)=\frac{r}{1-r} $, we can obtain $\Psi _{m,\mathbb{N}}^p \le 0$ for $r \le R_2$, where $R_2$ is the unique root in (0,1) of the equation (12).
	\par
	Next, we will show that the number $R_2$ is optimal.
	Let $z_0\in \partial B_X$ and $\mathbf{u}=(u_1,\dots,u_n)\in \partial \mathbb{U}^n$ with $|u_1|=\cdots=|u_n|=1$ be fixed. For $\lambda\in(0,1)$, let
	\begin{equation*}
		F(z)=f(l_{z_0}(z))\mathbf{u},\quad z\in B_X,
	\end{equation*}
	where
	\begin{equation*}
		f(\zeta)=\frac{\lambda-\zeta}{1-\lambda\zeta},\quad \zeta\in \mathbb{U}
	\end{equation*}
	and $l_{z_0}\in T(z_0)$.
	Let $v(z)=-l_{z_0}(z)^{m-1}z$, then we have
	\begin{align*}
		\left \| F(v(z)) \right \|_{\infty}^p &+\sum_{s=1}^{\infty}\left \| P_s(rz_0) \right \|_{\infty}+(\frac{1}{1+\left \| a \right \|_{\infty}} +\frac{r}{1-r}  )\left \| F_0 \right \|^2_{\infty} \\
		&=(\frac{\lambda+r^m }{1+\lambda r^m})^p+(1-\lambda^2)\frac{r}{1-\lambda r}+(\frac{1}{1+\lambda }\\ &+\frac{r}{1-r})(1-\lambda^2)^2\frac{r^2}{1-\lambda^2r^2}\\
		&=1+[(\frac{\lambda+r^m }{1+\lambda r^m})^p+(1-\lambda^2)\frac{r}{1-r}]\\
		&=1+(1-\lambda)\Psi  _{m}^p(\lambda,r)
	\end{align*}
	where
	\begin{equation*}
	\Psi  _{m}^p(\lambda,r)=\frac{1}{1-\lambda}[(\frac{\lambda+r^m }{1+\lambda r^m})^p]-1+(1+\lambda)\frac{r}{1-r}
	\end{equation*}
	Let $r \in (R_2, 1)$ be arbitrarily fixed. Since $R_2$ satisfies the equation (12), it follows that
	\begin{equation*}
	-p\left(\frac{1-r^m}{1+r^m}\right) + \frac{2r^N}{1-r} > 0
	\end{equation*}
	for $r \in (R_2, 1)$. Therefore
	\begin{equation*}
	\lim_{\lambda\to 1^-} \Psi_{m}^p(\lambda,r)
	= -p\left(\frac{1-r^m}{1+r^m}\right) + \frac{2r^N}{1-r}
	> 0,
    \end{equation*}
	and hence there exists $\lambda \in (0,1)$ such that $\Psi_{m}^p(\lambda,r) > 0$. This implies that
	\begin{equation*}
	\left \| F(v(z)) \right \|_{\infty}^p+\sum_{s=1}^{\infty}\left \| P_s(rz_0) \right \|_{\infty}+(\frac{1}{1+\left \| a \right \|_{\infty}} +\frac{r}{1-r}  )\left \| F_0 \right \|^2_{\infty} > 1.
    \end{equation*}
	Thus, $R_2$ is optimal. This completes the proof.
\end{proof}

\section{Improved Bohr-Rogosinski inequalities for holomorphic mappings with values in the closure of the unit ball of a $JB*-$triple}\label{sec4}
In this section, we will generalize the Bohr type inequality (8) and Bohr-Rogosinski type radius (11) to the case of holomorphic mappings on the unit ball of a complex Banach space with values in $\mathbb{B}_Y$.

\begin{theorem}
	Let $B_X$ be a unit ball of the complex banach space $X$ and let $\mathbb{B}_Y$ be the unit ball of a $\mathrm{JB}^*$-triple $Y$. Let $F$ be a holomorphic mapping from $B_X $ to $\overline{\mathbb{B}_Y}$ with
	\begin{align*}
		F(z)=a+\sum_{s=1}^{\infty} P_s(z),\quad z \in B_X,
	\end{align*}
	where $P_s(z)=\frac{1}{s!} D^sF(0)(z^s)$.
	Assume that $a\in \mathbb{B}_Y$. Then for $\|z\|=r \le R_3$, we have
	\begin{align}
\|a\|+\sum_{s=1}^{\infty}\frac{\|D\varphi_a(a)P_s(z)\|}{\|D\varphi_a(a)\|}+\Bigg(\frac{1}{1+\|a\|}+\frac{r}{1-r}\Bigg)\sum_{s=1}^{\infty} \Bigg[\frac{\|D\varphi_a(a)P_s(z)\|}{\|D\varphi_a(a)\|}\Bigg]^2 \le 1,
	\end{align}
the constant $R_3$ is the unique root with $r$ in $(0,1)$ of the equation
$$\|a\|+(1-\|a\|^2)\frac{r}{1-r}+(1-\|a\|^2)^2\Bigg(\frac{1}{1+\|a\|}+\frac{r}{1-r}\Bigg)\frac{r^2}{1-r^2}-1=0.$$
	The constant $R_3$ cannot be improved.
\end{theorem}
\begin{proof}
	Let $z_0=\frac{z}{\left \| z \right \| } \in B_X$ be fixed and let $f(\zeta)=F(\zeta z_0), \,\,\zeta\in \mathbb{U} $.
	Then $f(\zeta)\in H(\mathbb{U},\overline{\mathbb{B}_Y} )$ and
	\begin{equation*}
		f(\zeta)=a+\sum_{s=1}^{\infty} (P_s)(z_0)\zeta^s, \quad \zeta\in \mathbb{U}.
	\end{equation*}
	Let $a=f(0)\in \mathbb{B}_Y$, by Lemma 7,we have
	\begin{equation}
	\frac{\|D\varphi_a(a)P_s(z_0)\|}{\|D\varphi_a(a)\|}\le 1-\|a\|^2, \quad s\ge 1.
	\end{equation}
	Let $\|a\|=x\in [0,1)$ and $\|z\|=r$. By using Lemma 1 and estimates (16), we have
	\begin{align}
	\|a\|&+\sum_{s=1}^{\infty}\frac{\|D\varphi_a(a)P_s(rz_0)\|}{\|D\varphi_a(a)\|}+\Bigg(\frac{1}{1+\|a\|} +\frac{r}{1-r}\Bigg)\sum_{s=1}^{\infty}\Bigg[\frac{\|D\varphi_a(a)P_s(rz_0)\|}{\|D\varphi_a(a)\|}\Bigg]^2\notag\\
		&\le \|a\|+(1-\|a\|^2)\frac{r}{1-r}+(1-\|a\|^2)^2\Bigg(\frac{1}{1+\|a\|} +\frac{r}{1-r}\Bigg)\frac{r^2}{1-r^2}\notag\\
		&=1+x+(1-x^2)\frac{r}{1-r}+(1-x^2)^2\Bigg(\frac{1}{1+x} +\frac{r}{1-r}\Bigg)\frac{r^2}{1-r^2}-1\notag\\
&=1+\Psi_x (r),
	\end{align}
where $$\Psi_x(r)=x+(1-x^2)\frac{r}{1-r}+(1-x^2)^2\Bigg(\frac{1}{1+x} +\frac{r}{1-r}\Bigg)\frac{r^2}{1-r^2}-1.$$ The inequality is less than or equal to $1$ if and only if $\Psi_x (r) \le 0$. Obviously, $\Psi_x(r)$ is a continuous monotonically increasing function with $r$, and we note that
\begin{equation}
\lim\limits_{x\rightarrow 0^+}\Psi_x(r)=x-1<0,
\end{equation}
\begin{equation}
\lim\limits_{r\rightarrow 1^-}\Psi_x(r)=+\infty,
\end{equation}
which imply that there is a  unique $R_3\in(0,1)$ such that $\Psi_x(R_3)=0$. It is also easy to see that $R_3$ is monotonically decreasing with $x\in(0,1)$.
Furthermore, let
\begin{equation}
\lim\limits_{x\rightarrow 0^+}\Psi_x(r)=\frac{r}{1-r}+\Bigg(1+\frac{r}{1-r}\Bigg)\frac{r^2}{1-r^2}-1=0,
\end{equation}
we can obtain that $R_3=\frac{1}{3}+\sqrt[3]{\frac{1}{36}+\sqrt{\frac{257}{1296}}}+\sqrt[3]{\frac{1}{36}-\sqrt{\frac{257}{1296}}}\approx 0.4030$ with the condition $x\rightarrow 0^+$. Thus, for all $x\in(0,1)$, we can make sure that $R_3<\frac{1}{3}+\sqrt[3]{\frac{1}{36}+\sqrt{\frac{257}{1296}}}+\sqrt[3]{\frac{1}{36}-\sqrt{\frac{257}{1296}}}\approx 0.4030$.

Next, we will show that the number $R_3$ is optimal.
	Let $z_0\in \partial B_X$ and $\mathbf{u}\in \partial \mathbb{B}_Y$ be fixed. For $\lambda\in(0,1)$, let
	\begin{equation*}
		F(z)=f(l_{z_0}(z))\mathbf{u},\quad z\in B_X,
	\end{equation*}
	where
	\begin{equation*}
		f(\zeta)=\frac{\lambda-\zeta}{1-\lambda\zeta},\quad \zeta\in \mathbb{U}
	\end{equation*}
	and $l_{z_0}\in T(z_0)$. Then we have
	\begin{align*}
	\|a\|&+\sum_{s=1}^{\infty} \frac{\|D\varphi_a(a)P_s(rz_0)\|}{\|D\varphi_a(a)\|}+\Bigg(\frac{1}{1+\|a\|} +\frac{r}{1-r}\Bigg)\sum_{s=1}^{\infty} \Bigg[\frac{\|D\varphi_a(a)P_s(rz_0)\|}{\|D\varphi_a(a)\|}\Bigg]^2\\
		&=\lambda+(1-\lambda^2)\frac{r}{1-\lambda r}+\Bigg(\frac{1}{1+\lambda} +\frac{r}{1-r}\Bigg)(1-\lambda^2)^2\frac{r^2}{1-\lambda^2r^2} \\
		&=1+\lambda+(1-\lambda^2)\frac{r}{1-\lambda r}+\Bigg(\frac{1}{1+\lambda} +\frac{r}{1-r}\Bigg)(1-\lambda^2)^2\frac{r^2}{1-\lambda^2r^2}-1.\\
&=1+\Phi(\lambda,r),
	\end{align*}
where
\begin{equation}
\Phi(\lambda,r)=\lambda+(1-\lambda^2)\frac{r}{1-\lambda r}+\Bigg(\frac{1}{1+\lambda} +\frac{r}{1-r}\Bigg)(1-\lambda^2)^2\frac{r^2}{1-\lambda^2r^2}-1.
\end{equation}
Let $r=\frac{1}{3}+\sqrt[3]{\frac{1}{36}+\sqrt{\frac{257}{1296}}}+\sqrt[3]{\frac{1}{36}-\sqrt{\frac{257}{1296}}}\approx 0.4030\in(R_3,1)$ be fixed. From (), then $r>R_3$. We can computer
that $$\Phi\Bigg(0.90,\frac{1}{3}+\sqrt[3]{\frac{1}{36}+\sqrt{\frac{257}{1296}}}+\sqrt[3]{\frac{1}{36}-\sqrt{\frac{257}{1296}}}\Bigg)\approx 0.0283>0.$$
Thus, there exists $\lambda=0.9\in(0,1)$ such that $\Phi(\lambda,r)>1$ when $r>R_3$. Thus, $R_3$ is optimal. This completes the proof.
\end{proof}

\begin{theorem}
	Let $B_X$ is a unit ball of the complex banach space $X$, $\mathbb{B}_Y$ is the unit ball of a $\mathrm{JB}^*$-triple $Y$. Let F is a holomorphic mapping from $B_X $ to $\mathbb{B}_Y$ with the form	
	\begin{align*}
		F(z)=a+\sum_{s=1}^{\infty} P_s(z),\quad z \in B_X,
	\end{align*}
	where $P_s(z)=\frac{1}{s!} D^sF(0)(z^s)$.
    Assume that $a \in \mathbb{B}_Y$.Then for any $p \in (0,2]$, we have
	\begin{align}
	\| F(v(z))\|^p +\sum_{s=1}^{\infty}\frac{\|D\varphi_a(a)P_s(z)\|}{\|D\varphi_a(a)\|}+\Bigg(\frac{1}{1+\|a\|}+\frac{r}{1-r}\Bigg)
\sum_{s=1}^{\infty}\Bigg[\frac{\|D\varphi_a(a)P_s(z)\|}{\|D\varphi_a(a)\|}\Bigg]^2\le 1.
	\end{align}
for $\|z\|=r \le R_4$,
	where $v(z): B_X \to B_X$ is a Schwarz mapping having $z=0$ as a zero of order $m$, $\varphi_a \in \operatorname{Aut}(\mathbb{B}_Y)$ such that $\varphi_a(a)=0$ and $R_4$ is the unique root in $(0,1)$ of the equation (12). The number $R_4$ cannot be improved when $m=1$, $p=1$ or $p=2$.
    \end{theorem}
    \begin{proof}
	Let $z_0=\frac{z}{\left \| z \right \| } \in B_X$ be fixed and let $f(\zeta)=F(\zeta z_0), \,\,\zeta\in \mathbb{U} $.
	Then $f(\zeta)\in H(\mathbb{U},\overline{\mathbb{B}_Y} )$ and
	\begin{equation*}
		f(\zeta)=a+\sum_{s=1}^{\infty} (P_s)(z_0)\zeta^s, \quad \zeta\in \mathbb{U}.
	\end{equation*}
	Let $a=f(0)\in \mathbb{B}_Y$, by Lemma 7,we have
	\begin{equation}
	\frac{\|D\varphi_a(a)P_s(z_0)\|}{\|D\varphi_a(a)\|}\le 1-\|a\|^2, \quad s\ge 1.
	\end{equation}
	Let $\|a\|=x\in [0,1)$ and $\|z\|=r$. By using Lemma 1 and estimates (16), we have
	\begin{align}
	&\| F(v(rz_0))\|^p+\sum_{s=1}^{\infty}\frac{\|D\varphi_a(a)P_s(rz_0)\|}{\|D\varphi_a(a)\|}+\Bigg(\frac{1}{1+\|a\|} +\frac{r}{1-r}\Bigg)\sum_{s=1}^{\infty}\Bigg[\frac{\|D\varphi_a(a)P_s(rz_0)\|}{\|D\varphi_a(a)\|}\Bigg]^2\notag\\
		&\le\Bigg(\frac{\|a\|+r^m}{1+\|a\|r^m}\Bigg)^p+(1-\|a\|^2)\frac{r}{1-r}+(1-\|a\|^2)^2\Bigg(\frac{1}{1+\|a\|} +\frac{r}{1-r}\Bigg)\frac{r^2}{1-r^2}\notag\\
		&=1+\Bigg(\frac{x+r^m}{1+xr^m}\Bigg)^p+(1-x^2)\frac{r}{1-r}+(1-x^2)^2\Bigg(\frac{1}{1+x} +\frac{r}{1-r}\Bigg)\frac{r^2}{1-r^2}-1\notag\\
&=1+\Psi_m(x,r),
	\end{align}
where $$\Psi_m(x,r)=\Bigg[\Bigg(\frac{x+r^m}{1+xr^m}\Bigg)^p-1\Bigg]+(1-x^2)\Bigg[\frac{r}{1-r}+(1-x^2)\Bigg(\frac{1}{1+x} +\frac{r}{1-r}\Bigg)\frac{r^2}{1-r^2}\Bigg].$$ The inequality is less than or equal to $1$ if and only if $\Psi_m(x,r) \le 0$.

Taking $\varphi_0(r)=1$ and $$N(r)=\Bigg[\frac{r}{1-r}+(1-x^2)\Bigg(\frac{1}{1+x} +\frac{r}{1-r}\Bigg)\frac{r^2}{1-r^2}\Bigg]$$ in Lemma, we obtain $\Psi_m(x,r)\leq 0$ for $r\leq R_4$, where $R_4$ is the unique root in $(0,1)$ of the equation
\begin{equation}
\Phi_{p,m}(r)=p\Bigg(\frac{1-r^m}{1+r^m}\Bigg)-2\Bigg[\frac{r}{1-r}+(1-x^2)\Bigg(\frac{1}{1+x} +\frac{r}{1-r}\Bigg)\frac{r^2}{1-r^2}\Bigg]=0.
\end{equation}
In fact, $\Phi_{p,m}(r)$ is a continuous monotonically decreasing function with $r$, and we note that
\begin{equation}
\lim\limits_{r\rightarrow 0^+}\Phi_{p,m}(r)=p>0,
\end{equation}
\begin{equation}
\lim\limits_{r\rightarrow 1^-}\Phi_{p,m}(r)=-\infty,
\end{equation}
which imply that there is a  unique $R_4\in(0,1)$ such that $\Phi_{p,m}(R_4)=0$.

Next, we will show that the number $R_4$ is optimal when $m=1$, $p=1$ or $p=2$. From (), we know that
\begin{equation}
\Phi_{1,1}(r)=\Bigg(\frac{1-r}{1+r}\Bigg)-2\Bigg[\frac{r}{1-r}+(1-x^2)\Bigg(\frac{1}{1+x} +\frac{r}{1-r}\Bigg)\frac{r^2}{1-r^2}\Bigg].
\end{equation}
If $x=0$ (resp. $x=1$), we can obtain the $R_4\approx 0.2108$ (resp. $R_4=\sqrt{5}-2\approx 0.236)$ by using $\Phi_{1,1}(r)=0$. By (), it is easy to prove that
$r$ is decreasing function with $x$, thus, we can know that $R_4\in(0.2108, \sqrt{5}-2)$. Similarly, we also obtain $R_4\in\bigg(1-\frac{\sqrt{2}}{2}, \frac{1}{3}\bigg)$ when
$m=1$ and $p=2$.

Let $z_0\in \partial B_X$ and $\mathbf{u}\in \partial \mathbb{B}_Y$ be fixed. For $\lambda\in(0,1)$, let
	\begin{equation*}
		F(z)=f(l_{z_0}(z))\mathbf{u},\quad z\in B_X,
	\end{equation*}
	where
	\begin{equation*}
		f(\zeta)=\frac{\lambda-\zeta}{1-\lambda\zeta},\quad \zeta\in \mathbb{U}
	\end{equation*}
	and $l_{z_0}\in T(z_0)$.
	Let $v(z)=-l_{z_0}(z)^{m-1}z$, then we have
	\begin{align*}
&\| F(v(rz_0))\|^p+\sum_{s=1}^{\infty}\frac{\|D\varphi_a(a)P_s(rz_0)\|}{\|D\varphi_a(a)\|}+\Bigg(\frac{1}{1+\|a\|} +\frac{r}{1-r}\Bigg)\sum_{s=1}^{\infty}\Bigg[\frac{\|D\varphi_a(a)P_s(rz_0)\|}{\|D\varphi_a(a)\|}\Bigg]^2\notag\\
		&=\Bigg(\frac{\lambda+r^m }{1+\lambda r^m}\Bigg)^p+(1-\lambda^2)\frac{r}{1-\lambda r}+\Bigg(\frac{1}{1+\lambda }+\frac{r}{1-r}\Bigg)(1-\lambda^2)^2\frac{r^2}{1-\lambda^2r^2}\\
		&=1+\Bigg[\Bigg(\frac{\lambda+r^m }{1+\lambda r^m}\Bigg)^p-1\Bigg]+(1-\lambda^2)\frac{r}{1-\lambda r}+\Bigg(\frac{1}{1+\lambda }+\frac{r}{1-r}\Bigg)(1-\lambda^2)^2\frac{r^2}{1-\lambda^2r^2}.\\
&=1+\Psi_{p,m}(\lambda,r),
\end{align*}
where
\begin{equation}
\Psi_{p,m}(\lambda,r)=\Bigg[\Bigg(\frac{\lambda+r^m }{1+\lambda r^m}\Bigg)^p-1\Bigg]+(1-\lambda^2)\frac{r}{1-\lambda r}+\Bigg(\frac{1}{1+\lambda }+\frac{r}{1-r}\Bigg)(1-\lambda^2)^2\frac{r^2}{1-\lambda^2r^2}.
\end{equation}
Case I: If $m=p=1$, $r=\sqrt{5}-2>R_4$, then
\begin{equation}
\Psi_{1,1}(0.85,\sqrt{5}-2)=0.00136>0.
\end{equation}
From (), we obtain for $r \in (R_4, 1)$, there exists $\lambda \in (0,1)$ such that $\Psi_{p,m}(\lambda,r) > 0$. This implies that
	\begin{equation*}
	\| F(v(rz_0))\|^p+\sum_{s=1}^{\infty}\frac{\|D\varphi_a(a)P_s(rz_0)\|}{\|D\varphi_a(a)\|}+\Bigg(\frac{1}{1+\|a\|} +\frac{r}{1-r}\Bigg)\sum_{s=1}^{\infty}\Bigg[\frac{\|D\varphi_a(a)P_s(rz_0)\|}{\|D\varphi_a(a)\|}\Bigg]^2> 1.
	\end{equation*}\\
Case II: If $m=1$, $p=2$, $r=0.35>\frac{1}{3}>R_4$, then
\begin{equation}
\Psi_{2,1}(0.7902,0.35)>0.
\end{equation}
From (), we obtain for $r \in (R_4, 1)$, there exists $\lambda \in (0,1)$ such that $\Psi_{p,m}(\lambda,r)> 0$. This implies that
	\begin{equation*}
	\| F(v(rz_0))\|^p+\sum_{s=1}^{\infty}\frac{\|D\varphi_a(a)P_s(rz_0)\|}{\|D\varphi_a(a)\|}+\Bigg(\frac{1}{1+\|a\|} +\frac{r}{1-r}\Bigg)\sum_{s=1}^{\infty}\Bigg[\frac{\|D\varphi_a(a)P_s(rz_0)\|}{\|D\varphi_a(a)\|}\Bigg]^2> 1.
	\end{equation*}
Thus, $R_4$ is optimal when $m=1$, $p=1$ or $p=2$. This completes the proof.
\end{proof}

\section*{\rm{\bf Acknowledgements}}
Liangpeng Xiong is supported by the National Natural Science Foundation of China
(No.12061035) and Research Foundation of Jiangxi Science and Technology
Normal University of China (No.2021QNBJRC003). Junzhou Xiong is supported by Research Foundation of Jiangxi Provincial Education Department (No.YC2024-S689).


%%       \citet{<label>} ==> Jones et al. [21]
%%       \citep{<label>} ==> [21]
%%


\begin{thebibliography}{00}
\bibitem{ref1}Abu-Muhanna Y, Ali R M, Ng Z C, Hasni S F M.: Bohr radius for subordinating families of analytic functions and bounded harmonic mappings. Journal of Mathematical Analysis and Applications, \textbf{420}, 124--136 (2014)
	
	\bibitem{ref2}Aizenberg L.: Generalization of results about the Bohr radius for power series. Studia Mathematica, \textbf{180}, 161--168 (2007)
	
	\bibitem{ref3}Aizenberg L.: Multidimensional analogues of Bohr's theorem on power series. Proceedings of the American Mathematical Society, \textbf{128}(4), 1147--1155 (2000)
	
	\bibitem{ref4}Aizenberg L, Aytuna A, Djakov P.: Generalization of a theorem of Bohr for bases in spaces of holomorphic functions of several complex variables. Journal of Mathematical Analysis and Applications, \textbf{258}(1), 429--447 (2001)
	
	\bibitem{ref5}Alkhaleefah S A, Kayumov I R, Ponnusamy S.: On the Bohr inequality with a fixed zero coefficient. Proceedings of the American Mathematical Society, \textbf{147}, 5263--5274 (2019)
	
	\bibitem{ref6}Bayart F, Pellegrino D, Seoane-Seplveda J B.: The Bohr radius of the n-dimensional polydisk is equivalent to $(\log n)/n$. Advances in Mathematics, \textbf{264}, 726--746 (2014)
	
	\bibitem{ref7}Bnteau C, Dahlner A, Khavinson D.: Remarks on the Bohr phenomenon. Computational Methods and Function Theory, \textbf{4}(1), 1--19 (2004)
	
	\bibitem{ref8}Boas H P, Khavinson D.: Bohr's power series theorem in several variables. Proceedings of the American Mathematical Society, \textbf{125}(10), 2975--2979 (1997)
	
	\bibitem{ref9}Bohr H.: A theorem concerning power series. Proceedings of the London Mathematical Society, s2-\textbf{13}, 1--5 (1914)
	
	\bibitem{ref10}Chen K, Liu M S, Ponnusamy S.: Bohr-type inequalities for unimodular bounded analytic functions. Results in Mathematics, \textbf{78}(5), Paper No.183, 16 pp (2023)
	
	\bibitem{ref11}Chen S L, Hamada H, Ponnusamy S, Vijayakumar R.: Schwarz type lemmas and their applications in Banach spaces. Journal d'Analyse Mathematique, \textbf{152}, 181--216 (2024)
	
	\bibitem{ref12}Chu C H.: Jordan structures in geometry and analysis. Cambridge University Press, Cambridge (2012). (Cambridge Tracts in Mathematics 190)
	
	\bibitem{ref13}Chu C H.: Bounded symmetric domains in Banach spaces. World Scientific Publishing Co. Pte. Ltd., Hackensack, NJ (2021)
	
	\bibitem{ref14}Defant A, Frerick L.: A logarithmic lower bound for multi-dimensional Bohr radii. Israel Journal of Mathematics, \textbf{152}, 17--28 (2006)
	
	\bibitem{ref15}Defant A, Frerick L.: The Bohr radius of the unit ball of $\ell_p^n$. Journal for die reine und angewandte Mathematik, \textbf{660}, 131--147 (2011)
	
	\bibitem{ref16}Defant A, Frerick L, Ortega-Cerd J, Ounaes M, Seip K.: The Bohnenblustille inequality for homogeneous polynomials is hypercontractive. Annals of Mathematics, \textbf{174}(1), 485--497 (2011)
	
	\bibitem{ref17}Dineen S, Timoney R M.: Absolute bases, tensor products and a theorem of Bohr. Studia Mathematica, \textbf{94}(3), 227--234 (1989)
	
	\bibitem{ref18}Djakov P B, Ramanujan M S.: A remark on Bohr theorem and its generalizations. Journal of Analysis, \textbf{8}, 65--77 (2000)
	
	\bibitem{ref19}Graham I, Kohr G.: Geometric Function Theory in One and Higher Dimensions. Marcel Dekker Inc., New York (2003)
	
	\bibitem{ref20}Hamada H.: Bohr phenomenon for analytic functions subordinate to starlike or convex functions. Journal of Mathematical Analysis and Applications, \textbf{499}, 125019 (2021)
	
	\bibitem{ref21}Hamada H, Honda T.: Some generalizations of Bohr's theorem. Mathematical Methods in the Applied Sciences, \textbf{35}(17), 2031--2035 (2012)
	
	\bibitem{ref22}Hamada H, Honda T.: Bohr phenomena for holomorphic mappings with values in several complex variables. Results in Mathematics, \textbf{79}, Article 239 (2024)
	
	\bibitem{ref23}Hamada H, Honda T.: Bohr radius for pluriharmonic mappings in separable complex Hilbert spaces. Bulletin of the Malaysian Mathematical Sciences Society, \textbf{47}, Article 47 (2024)
	
	\bibitem{ref24}Hamada H, Honda T, Kohr G.: Bohr's theorem for holomorphic mappings with values in homogeneous balls. Israel Journal of Mathematics, \textbf{173}, 177--187 (2009)
	
	\bibitem{ref25}Hamada H, Honda T, Kohr M.: Bohr-Rogosinski radius for holomorphic mappings with values in higher dimensional complex Banach spaces. Analysis and Mathematical Physics, \textbf{15}, Article 64 (2025)
	
	\bibitem{ref26}Hamada H, Honda T, Mizota Y.: Bohr phenomenon on the unit ball of a complex Banach space. Mathematical Inequalities \& Applications, \textbf{23}(4), 1325--1341 (2020)
	
	\bibitem{ref27}Huang Y, Liu M S, Ponnusamy S.: Refined Bohr-type inequalities with area measure for bounded analytic functions. Analysis and Mathematical Physics, \textbf{10}, 50 (2020)
	
	\bibitem{ref28}Kayumov I R, Khammatova D M, Ponnusamy S.: Bohr-Rogosinski phenomenon for analytic functions and Cesro operators. Journal of Mathematical Analysis and Applications, \textbf{496}(2), 124824, 17 pp (2021)
	
	\bibitem{ref29}Kayumov I R, Ponnusamy S.: On a powered Bohr inequality. Annales Academi Scientiarum Fennic Mathematica, \textbf{44}, 301--310 (2019)
	
	\bibitem{ref30}Kaup W.: A Riemann mapping theorem for bounded symmetric domains in complex Banach spaces. Mathematische Zeitschrift, \textbf{183}, 503--529 (1983)
	
	\bibitem{ref31}Krantz S G.: Geometric function theory: Explorations in complex analysis. Birkhuser, Boston (2006)
	
	\bibitem{ref32}Liu G, Liu Z, Ponnusamy S.: Refined Bohr inequality for bounded analytic functions. Bulletin des Sciences Mathematiques, \textbf{173}, 103054 (2021)
	
	\bibitem{ref33}Liu M S, Ponnusamy S.: Multidimensional analogues of refined Bohr inequality. Proceedings of the American Mathematical Society, \textbf{149}(5), 2133--2146 (2021)
	
	\bibitem{ref34}Paulsen V I, Popescu G, Singh D.: On Bohr's inequality. Proceedings of the London Mathematical Society, \textbf{85}, 493--512 (2002)
	
	\bibitem{ref35}Ponnusamy S, Vijayakumar R, Wirths K-J.: New inequalities for the coefficients of unimodular bounded functions. Results in Mathematics, \textbf{75}, 107 (2020)
	
	\bibitem{ref36}Ponnusamy S, Vijayakumar R, Wirths K J.: Improved Bohr phenomenon in quasi-subordination classes. Journal of Mathematical Analysis and Applications, \textbf{506}(1), Paper No.125645, 10 pp (2022)
	
	\bibitem{ref37}Rogosinski W.: ber Bildschranken bei Potenzreihen und ihren Abschnitten. Mathematische Zeitschrift, \textbf{17}, 260--276 (1923)
































\bibitem{BFG}R. W. Barnard, C. H. Fitzgerald, S. Gong, A distortion theorem of biholomorphic convex mappings in $\mathbb{C^{\mathrm{2}}}$, Trans. Am. Math. Soc. 344 (1994) 907--924.

\bibitem{CHHK}
C. H. Chu, H. Hamada, T. Honda, G. Kohr, Distortion theorems for convex mappings on homogeneous balls, J. Math. Anal. Appl. 369 (2010) 437--442.

\bibitem{DLP}R.D{\l}ugosz, P.Liczberski, Some results of Fekete-Szeg\"{o} type for Bavrin's families of holomorphic functions in $\mathbb {C}^{n}$, Ann. Mat. Pura Appl.
200 (2021)1841--1857.

\bibitem{G}S.Gong, Convex and starlike mappings in several complex variables (2nd ed), Beijing: Science Press, 2003.

\bibitem{GWY}	
S.Gong, S. K. Wang, Q. H. Yu, Biholomorphic convex mappings of unit ball in $\mathbb{C^{\mathit{n}}}$, Pac. J. Math. 161 (1993)287--306.

\bibitem{GL}
S.Gong, T. S. Liu, Distortion theorems for biholomorphic convex mappings on bounded convex circular domains, Chin. Ann. Math. 20B (1999)297-304.

\bibitem{GV}
I.Graham, D.Varolin, Bloch constants in one and several variables, Pac. J. Math. 174(2) (1996) 347--357.

\bibitem{GK}
I.Graham, G.Kohr, Geometric function theory in one and higher dimensions, New York(NY): Marcel Dekker, 2003.

\bibitem{GKS}
S.Giri, S.S.Kumar, Toeplitz determinants for a class of holomorphic mappings in higher dimensions, Complex Anal. Oper. Theory 17(86) (2023) https://doi.org/10.1007/s11785-023-01394-0.

\bibitem{HK}
H. Hamada, G.Kohr, Growth and distortion results for convex mappings in infinite dimensional spaces, Complex Variables 47 (2002) 291--301.

\end{thebibliography}


\end{document}